\documentclass[11pt, reqno, english]{amsart}  
\usepackage[utf8]{inputenc}
\usepackage{amsthm}
\usepackage{amssymb}
\usepackage{amsmath}
\usepackage{graphicx}
\usepackage{tikz}
\usepackage{mathrsfs}
\usepackage{float}
\usepackage{makecell}
\usepackage{amscd}
\usepackage{multirow}
\usepackage{xspace}
\usepackage{verbatim}
\usepackage{fontenc}
\usepackage{hhline}
\usepackage{hyperref}
\hypersetup{colorlinks,linkcolor=black,filecolor=black,urlcolor=black,citecolor=black}

\usepackage{cite}
\usepackage[left=1in,right=1in,top=1in,bottom=1in]{geometry}

\usepackage[left=1in,right=1in,top=1in,bottom=1in]{geometry}

\newtheorem{theorem}{Theorem}
\newtheorem{lemma}[theorem]{Lemma}

\newtheorem{question}[theorem]{Question}
\newtheorem{corollary}[theorem]{Corollary}
\newtheorem{definition}[theorem]{Definition}
\newtheorem{observation}[theorem]{Observation}

\newcommand{\Z}{\mathbb{Z}}

\newcommand{\R}{\mathbb{R}}

\newcommand{\Gr}{\text{Gr}}
\renewcommand{\epsilon}{\varepsilon}

\newcommand{\cpoints}{\rho}
\newcommand{\cplane}{\nu}

\newcommand{\p}{\mathbf{p}}
\newcommand{\q}{\mathbf{q}}

\title{Youden's Demon is Sylvester's Problem}
\author{Florian Frick \and Andrew Newman \and Wesley Pegden}
\address[FF]{Dept. Math. Sciences, Carnegie Mellon University, Pittsburgh, PA 15213, USA}
\email{frick@cmu.edu}
\address[AN]{Dept. Math. Sciences, Carnegie Mellon University, Pittsburgh, PA 15213, USA}
\email{anewman@andrew.cmu.edu}
\address[WP]{Dept. Math. Sciences, Carnegie Mellon University, Pittsburgh, PA 15213, USA}
\email{wes@math.cmu.edu}


\date{\today}

\thanks{FF was supported by NSF CAREER Grant DMS 2042428. WP was supported by NSF grant DMS 1700365.}

\begin{document}

\maketitle
\begin{abstract} If four people with Gaussian-distributed heights stand at Gaussian positions on the plane, the probability that there are exactly two people whose height is above the average of the four is exactly the same as the probability that they stand in convex position; both probabilities are $\frac 6 \pi \arcsin\left(\frac{1}{3}\right)\approx .649$.  We show that this is a special case of a more general phenomenon: The problem of determining the position of the mean among the order statistics of Gaussian random points on the real line (``Youden's demon problem'') is the same as a natural generalization of Sylvester's Four Point problem to Gaussian points in $\R^d$.  Our main tool is the observation that the Gale dual of independent samples in $\R^d$ itself can be taken to be a set of independent points (conditioned on barycenter at the origin) when the distribution of the points is Gaussian.
\end{abstract}

\section{Introduction}
Sylvester's Four Point problem, first posed in \cite{Sylvester} in 1864, is often said to be the starting point for the study of random polytopes. In its original formulation, Sylvester's Four Point problem asks for the probability that four randomly selected points in $\R^2$ are in convex position. Obviously though this probability will depend on the underlying distribution. The early history of this problem is outlined in \cite{Pfiefer}, but the short version is that early on this question was considered for the uniform distribution on a convex region $K$ in the plane. In this direction Sylvester himself showed that the probability in question is 2/3 if $K$ is a triangle; Woolhouse showed that the answer is $1 - \frac{35}{12\pi^2} \approx .704$ if $K$ is a disk, and then Blaschke \cite{Blaschke} proved that among convex regions in the plane the triangle and the disk respectively minimize and maximize the answer to Sylvester's question. Generalization of Sylvester's problem extending the number of points or the dimension of the ambient space or both have been considered in the years since, for example \cite{BaranyPlane} studied the question of $n$ uniform points in a convex body in $\R^2$ to be in convex position and Kingman~\cite{Kingman} gave an exact formula for the case of $d + 2$ points in $\R^d$ sampled uniformly at random from the $d$-ball. See Schneider's chapter~\cite{schneider2017} in the \emph{Handbook of Discrete and Computational Geometry} for historical context and details.

The question of Sylvester's problem for Gaussian points has apparently drawn less attention than the uniform case.  The only solution we are aware of is for the case of four points in $\R^2$, as obtained by Maehara \cite{Maehara}, and later Blatter \cite{Blatter} by a different proof.  They showed that the probability that four points distributed as standard Gaussians in the plane are in convex position is
\begin{equation}
\label{X2} \frac{6}{\pi} \arcsin\left( \frac{1}{3} \right) \approx .649.
\end{equation}

Naturally, Sylvester's problem can be generalized to the question of determining the probability that $d + 2$ random points in $\R^d$ are in convex position. But beyond this, for $d \geq 4$ we can consider the stronger question on the distribution of the {Radon partition} of the random points. We recall Radon's Theorem:

\begin{theorem}[Radon~\cite{Radon}]\label{t.radon}
Any set $S$ of $d + 2$ points in $\R^d$ can be partitioned into two sets $S_1$ and~$S_2$ so that the convex hull of $S_1$ intersects the convex hull of~$S_2$. If the points are in general position then the partition is unique up to swapping $S_1$ and~$S_2$.
\end{theorem}
\noindent The partitions of $n$ points into two sets whose convex hulls intersect are the \textbf{Radon partitions} of the point set. For $n \geq d + 2$, Theorem \ref{t.radon} guarantees at least one Radon partition. The other partitions of the $n$ points are called the \textbf{affine separations} and are exactly the ways of dividing the points into two subsets so that there is an affine hyperplane between them.

From the view of Radon's Theorem, the question of whether $d + 2$ points lie in convex position in $\R^d$ is the same as asking whether the Radon partition is \emph{not} a $(1, d + 1)$ split. A natural extension of Sylvester's problem in higher dimensions is thus to to ask about the full distribution on the splits:
\begin{question}
Let $X_d$ be the random variable counting the number of points in the smaller side of the Radon partition of $d + 2$ standard Gaussian points in $\R^d$. What is the distribution of $X_d$?
\end{question}

The table below summarizes, for small values of $d$, an empirical distribution of $X_d$ across 1000 samples each.

\begin{table}[h!]
\centering
\caption{Empirical Distribution for \(X_d\)}
\begin{tabular}{|c|c|c|c|c|}
\hline
$d \downarrow$, $X_d \rightarrow$  & 1 & 2 & 3 & 4 \\ \hline
2 & .341 & .659 & - & - \\ \hline
3 & .091 & .909 & - & - \\ \hline
4 & .031 & .460 & .509 & - \\ \hline
5 & .008 & .180 & .812 & - \\ \hline
6 & .000 & .074 & .475 & .451 \\ \hline
\end{tabular}
\end{table}
As expected for $d = 2$, the empirical distribution is in line with Maehara's result. Is it possible to find exact values for the true distribution for higher values of $d$?

In what might initially seem like an unrelated question, \textbf{Youden's Demon} problem in statistics is to determine the probability that the sample mean $\bar Z$ among $n$ i.i.d. samples $Z_1,\dots,Z_n$ lies between the the $k$th and $(k+1)$st largest (or smallest) sample.  Letting $P(n,k)$ be the probability that the mean $\bar Z$ lies between the $k$th largest and $(k+1)$st largest samples $Z_k,Z_{k+1}$ in the Gaussian case $Z_i\sim N(0,1)$, David in 1963 \cite{David} proved the following:
\begin{theorem}[David~\cite{David}]
    \begin{align}
        \label{P42} P(4,2) &= \frac{6}{\pi} \arcsin\left( \frac{1}{3} \right) \approx .649.\\
        \label{P52} P(5,2) &= \frac{1}{4} + \frac{5}{2\pi} \arcsin \left(\frac{1}{4} \right)\approx .451.\\
        \label{Pnk} P(n,k)&\sim \left(\frac{k^{n-k-1}e^{n-2k}}{2(k!)^2 n^{n-3k-1}(2\pi)^{n-k}}\right)^{1/2}.
        %\frac{1}{2} + \frac{5}{\pi} \arcsin \left(\frac{1}{4} \right) \approx .902.
    \end{align}
\end{theorem}
\noindent Here $a_n\sim b_n$ indicates that 
$\lim\limits_{n\to \infty} \frac {a_n}{b_n}=1.$
\smallskip

Our main result is an exact equivalence between the distributions $X_d$ and $P(n,k)$, showing that the relationship between \eqref{P42} and \eqref{X2} is not coincidental:
\begin{theorem}\label{t.equiv}
    For any $d$ and $k \leq \lfloor \frac{d + 1}{2}\rfloor$ other than $k = \frac{d + 1}{2}$, 
    \[\Pr(X_d = k) = 2P(d + 2,k).\]
    If $k = \frac{d + 1}{2}$ then
    \[\Pr(X_d = k) = P(d + 2, k).\]
\end{theorem}
\noindent Note that $2P(d+2,k)=P(d + 2, k) +P(d + 2, d + 2 - k)$, so this is why the $(d + 1)/2$ case does not have this factor of 2.

Theorem \ref{t.equiv} allows us to immediately translate results for Youden's demon problem to Sylvester's problem for Gaussian points.  Apart from giving Maehara's and Blatter's result for the probability that four Gaussian points in $\R^2$ are in convex position, we obtain new results on the Gaussian Sylvester problem as follows:
\begin{corollary}
For five points drawn independently from the standard normal distribution in~$\R^3$, the probability that the points are in convex position is given by
    \begin{equation}
    \Pr(X_3=2)=\frac 1 2 + \frac 5 \pi \arcsin\left(\frac 1 4\right)\approx .902.
    \end{equation}
\end{corollary}
\begin{corollary}
For $d+2$ points drawn independently from the standard normal distribution in~$\R^d$, the probability that the points are in convex position, asymptotically in~$d$, satisfies
    \begin{equation}
    1-\Pr(X_d\neq 1)\sim 2\left(\frac{e^{d}}{2 (d+2)^{(d - 2)}(2\pi)^{d+1}}\right)^{1/2}.
    \end{equation}
%    \begin{equation}
%    \Pr(X_d\neq 1)\sim 1-2\left(\frac{e^{n-2}}{2 n^{n-4}(2\pi)^{n-1}}\right)^{1/2}
%    \end{equation}
\end{corollary}

We prove Theorem \ref{t.equiv} via Gale duality.  The \emph{Gale dual} of a sequence $\{x_i\}_{i=1}^n$ of $n$ points in $\R^d$ is a sequence $\{y_i\}_{i=1^n}$ of $n$ points in $\R^{n-(d+1)}$ (unique up to linear transformations) with the property that Radon partitions of the points $\{x_i\}$ correspond to \textbf{linear} separations of the points $\{y_i\}$, i.e. ways to split the points into two sets that are separated by a hyperplane that passes through the origin.  In general, a point $y_j$ in the dual depends not just on $x_j$ but on the whole sequence $x_1,\dots,x_n$.  But the key observation that gives rise to Theorem \ref{t.equiv} is that when the $\{x_i\}$ is a sequence of independent Gaussian points, the $\{y_i\}$'s can also be taken to be a sequence of independent Gaussians, conditioned just to have barycenter~$0$ (Lemma \ref{l.galegauss}).

This fact may be of more general interest; for example, Schneider \cite{Schneider} observed that neighborly polytopes can be generated not only by choosing their vertices as independent Gaussians, but also by choosing their Gale duals as independent samples from any even distribution $\phi$ assigning measure 0 to every hyperplane; when $\phi$ is a Gaussian distribution, our result gives a simple explanation for this phenomenon.  And in~\cite{Kuchelmeister}, Kuchelmeister connects linear separability for Gaussian points to Youden's demon for Gaussian points just by virtue of the commonalities in the calculations that determine each (in particular, computing volumes of spherical simplices); Lemma~\ref{l.galegauss} makes this connection clear in our particular case when we are comparing $X_d$ and~$P(n, k)$.

%Directly from work on Youden's Demon problem by David \cite{David}, Theorem \ref{t.equiv}  implies Maehara's result on $X_2$ and also settles the $d = 3$ case:
%\begin{corollary}
%\[\Pr(X_3 = 2) = \frac{1}{2} + \frac{5}{\pi} \arcsin \left(\frac{1}{4} \right) \approx .902.\]   \end{corollary}

%Youden's Demon is the problem of determining the position of the sample mean among order statistics sampled from some distribution. In our case, the distribution will be the standard Gaussian. For $Z_1, Z_2, ..., Z_n$ independent samples from $N(0, 1)$ with sample mean $\overline{Z} = \frac{Z_1 + \cdots + Z_n}{n}$. Following notation of David, let $P(n, k)$ be the probability that $\overline{Z}$ is between the $k$th smallest and the $(k + 1)$st smallest $Z_i$. Youden's Demon in this case is the following question:
%\begin{question}
%    What is the distribution of $P(n, k)$?
%\end{question}
%\commAN{Add some background on Youden's Demon.}
%Our main theorem is the following:

\section{Gale duality}
We begin with a proof of Radon's Theorem. This proof is simple, but worth including here as it sets up the key ideas from linear algebra, including Gale duality, that we will be using later.
\begin{proof}[Proof of Radon's Theorem]
Let $x_1, ..., x_{d + 2}$ be points in $\R^d$. Construct a $(d + 1) \times (d + 2)$ matrix~$M$ by taking initially the columns to be the $x_i$'s and then adding an additional row that contains all~$1$'s. Take $v$ a vector in the right kernel of~$M$. Partition the points into 
$A = \{x_i \mid v_i \geq 0\}$
and
$B = \{x_i \mid v_i < 0\}$
Then, because of the all $1$'s row, we have
\[\sum_{i \in A} v_i = \sum_{j \in B} -v_j\]
and so taking, 
\[\lambda_i = \frac{|v_i|}{\sum_{i \in A} v_i}\]
we have 
$\sum_{i \in A} \lambda_i = 1$ and $\sum_{h \in B} \lambda_j = 1$ with
$\sum_{i \in A} \lambda_i x_i = \sum_{j \in B} \lambda_j x_j$, so the convex hull of $A$ and $B$ intersect nontrivially. Moreover, if the points are in general position then $\text{null}(M) = 1$ and the sign pattern of $v$ is unique up to swapping to $-v$ with all coordinates of $v$ nonzero, thus the Radon partition is unique.
\end{proof}
From this proof we see that basic linear algebra gives us a geometric description of the full Radon partitions of an arbitrary number of points in $\R^d$. Given $n$ points in $\R^d$ in general position take $M$ to be the $(d + 1) \times n$ matrix $M$ obtained by taking the points as the columns vectors and then adding a row of all~$1$'s. The Radon partitions of the points are exactly the sign-patterns corresponding to the orthants of $\R^n$ intersected by right kernel of $M$. Take $N$ to be any matrix whose columns are a basis for $\ker(M)$. The \textbf{Gale dual} of the original point set are the points given by the \textbf{rows} of $N$. If we start with $n$ points in general position in $\R^d$, then the Gale dual is a set of $n$ points in $\R^{n - d - 1}$ with barycenter zero. Note that the Gale dual is unique only up to the choice of basis for $\ker(M)$. For more details on Gale duality we refer the reader to Matou\v sek's book for details~{\cite[Sec.~5.6]{matousek2013}}. The most useful fact about the Gale dual for us is that if $x_1, .., x_n$ are in general position in $\R^d$ and have Gale dual $y_1, ..., y_n$ in $\R^{n - d - 1}$ then the Radon partitions of the $x_i$'s correspond exactly to the linear separations of the $y_i$'s. Specifically, for $A \subseteq [n]$, $\{x_i \mid i \in A\} \sqcup \{x_j \mid i \in [n] \setminus A\}$ is a Radon partition if and only if $\{y_i \mid i \in A\} \sqcup \{y_j \mid j \in [n] \setminus A\}$ is a linear separation. 

It follows that for $n$ points in general position in $\R^d$ the number of Radon partitions is equal to the number of linear separations of $n$ points in general position in $\R^{n- d - 1}$. This turns out to have a specific value that depends only on $n$ and $d$, but not on how the points are arranged, other than a general position assumption. This enumeration result is known as Cover's theorem~\cite{Cover} in the statistics literature, and is independently due to Eckhoff~\cite{Eckhoff}.
\begin{theorem}%[Eckhoff~\cite{Eckhoff}]
For $n$ points in general position in $\R^d$, the number of Radon partitions is 
\[\sum_{i = d + 1}^{n - 1} \binom{n - 1}{i},\]
the number of affine separations is
\[\sum_{i = 0}^{d} \binom{n - 1}{i},\]
and the number of linear separations is
\[\sum_{i = 0}^{d - 1} \binom{n - 1}{i}.\]
\end{theorem}

The key lemma we need to prove our main result is that the Gale dual of Gaussian points in $\R^d$ can be taken to be independent Gaussian points in $\R^{n - d - 1}$ conditioned to have barycenter zero. With this lemma our main theorem follows almost immediately. Indeed $d + 2$ Gaussian points in $\R^d$ are Gale dual to $d + 2$ points on the real line which are Gaussian distributed conditioned on having barycenter zero. The Radon partition is determined by which points in the Gale dual are negative and which are positive. This means that $X_d$ is the minimum between the number of positive points in the Gale dual and the number of negative points in the Gale dual. This is the same as Youden's demon; determining the position of the mean among the order statistics. 

In Section \ref{GaussianDual} we provide the full details for this key lemma about the Gale dual of Gaussian points, but here we explain the geometric intuition. A $k$-dimensional linear subspace $M$ of $\R^n$ partitions the orthants of $\R^n$ into orthants intersected by $M$ and orthants intersected by the orthogonal complement of $M$. If we take a collection of $n$ points $v_1, ..., v_n$ in general position in $\R^d$ then the Radon partitions of the points correspond exactly to such a collection of orthants in $\R^n$. Indeed if we take a the matrix $A$ to be the $(d + 1) \times n$ matrix obtained by adding the all ones vector to the $d \times n$ matrix whose columns are $v_1, ..., v_n$ then the orthants spanned by $A^{T}$ correspond exactly to the affine splits of the $v_i$'s, and the orthants not spanned by $A^{T}$, and therefore necessarily spanned by the orthogonal complement, correspond to the Radon partitions. 

The Radon partitions of a set of $n$ points in $\R^d$ are then a property of a naturally associated linear subspace of $\R^n$. A natural way to sample Radon partitions for a random set of $n$ points in $\R^d$ would be to pick a uniform random linear subspace of $\R^n$ of dimension $d$ and then making it $d + 1$ dimensional by adding the all ones vector. When we pick the uniform random subspace, we are sampling with respect to the Haar measure on the real Grassmannian $\Gr_{d}(n)$. It is well known that the column span of an $n \times d$ matrix whose entries are iid standard Gaussians is a uniform random sample from $\Gr_{d}(n)$. Taking the orthogonal complement of a uniform random point in $\Gr_d(n)$ ought to be a uniform random point in $\Gr_{n - d}(n)$ which can itself be sampled as the columns span of an $n \times (n - d)$ matrix of iid standard Gaussians. Thus the orthogonal complement of a Gaussian matrix is also a Gaussian matrix. In the case of Radon partitions all that changes is that we add the all ones vector to the original space, and so this is why we should expect that the Gale dual of $n$ Gaussian points in $\R^d$ is $n$ points in $\R^{n - d  - 1}$ conditioned to have barycenter zero.

\section{Gale dual for independent Gaussians} \label{GaussianDual}
The Gale dual of $n$ points in $\R^d$ is $n$ points in $\R^{n - d - 1}$ that have barycenter zero. For our proof here we take the additional assumption that the original points also have barycenter zero. As this can be accomplished by a simple translation of the points, this will not affect the Radon partitions. 

Letting $I_n$ and $J_n$ denote the $n\times n$ identity and $n\times n$ all-ones matrix, respectively, we define:
\begin{definition}
Define $\cpoints_n^d$ to be a multivariate Gaussian distribution generated as  $\Z_{d\times n}A$, where $A=(I_n-\frac 1 n J_n)$. 
\end{definition}
\noindent The columns of a matrix from distribution $\cpoints_n^d$ are then in barycentric position.  Note that $\cpoints_n^d$ can be sampled by generating the the columns of $\Z_{d\times n}$ as independent standard Gaussians in $\R^d$, and then subtracting the barycenter of the collection of $n$ points from each.  The point of this section is to prove the following lemma:
\begin{lemma}\label{l.galegauss}
 There is a coupling $(\{x_i\}_{i=1}^n,\{y_i\}_{i=1}^n)$ of $\cpoints_n^d, \cpoints_n^{n-d-1}$ such that $\{x_i\}_{i=1}^n$ and $\{y_i\}_{i=1}^n$ are in Gale dual position.
\end{lemma}

Closely related to $\cpoints_n^d$ is the distribution on collections of Gaussian points normal to the all 1's vector $\mathbf{1}\in \R^d$:
\begin{definition}\label{d.cplane}
Define $\cplane_n^d$ to be a multivariate Gaussian distribution generated as $A\Z_{d\times n}$, where $A=(I_d-\frac 1 d J_d)$.
\end{definition}
In particular, $\cplane_n^d$ can be viewed as the distribution of $n$ \emph{independent} Gaussian samples in $\R^d$, each with covariance matrix $I_d-\frac 1 d J_d$ (and so each orthogonal to $\mathbf{1}\in \R^d$).  

\begin{observation}\label{symmorth1}
Viewed as a collection of $n$ points in $\R^d$, the distribution $\cplane_n^d$ is symmetric with respect to any orthogonal transformation fixing $\mathbf{1}$.\qed
\end{observation}

Lemma \ref{l.galegauss} is a consequence of the following:

\begin{lemma}\label{l.cplanecoupling}
    There is a coupling $(X,Y)$ of $\nu_n^d$ and $\nu_{d-n-1}^d$ so that the columns of $X$ are orthogonal to the columns of $Y$.
\end{lemma}

\noindent First we prove Lemma \ref{l.galegauss} from \ref{l.cplanecoupling}.

\begin{proof}[Proof of Lemma \ref{l.galegauss}]
For a sequence of points $\p_1,\dots,\p_n$ in $\R^d$, we define the points 
\[
\bar \p_1,\dots,\bar \p_d\in \R^n
\]
as the rows of the matrix whose columns are the points $\p_i$.  In particular, writing $x^i$ for the $i$th coordinate of the point $x$, we have that $\bar \p_i=\{p_1^i,\dots,p_n^i\}$.

Note that we have that 
\[
\{\p_i\}_{i=1}^n\sim \cpoints_n^d \iff \{\bar \p_j\}_{j=1}^d\sim \cplane_d^n,\]
Similarly, we have that\[
\{\q_i\}_{i=1}^n\sim \cpoints_n^{n-d-1}\iff \{\bar \q_j\}_{j=1}^{n-d-1}\sim \cplane_{n-d-1}^n.
\]
Thus, to prove the statement it suffices to give a coupling $(X,Y)$ of the distributions $\cplane_d^n, \cplane_{n-d-1}^n$, with the property that $X$ is always orthogonal to $Y$, and this coupling is given by Lemma \ref{l.cplanecoupling}.
\end{proof}

Our proof of Lemma \ref{l.cplanecoupling} we will give an explicit such coupling, which will thus give a constructive proof of Lemma \ref{l.galegauss} as well.  Informally: we simply generate $X\sim \nu_n^d$ and $W\sim \nu_{d-n-1}^d$, and then apply an orthogonal transformation to $W$ so that in the resulting matrix $Y$, all column vectors are orthogonal to all column vectors of $X$, as well as to $\mathbf{1}$, which is feasible since $n+(d-n-1)+1=n$.  By choosing the orthogonal transformation randomly from all possible such transformations, the orthogonal transformation is a uniformly random orthogonal transformation even if we condition on $W$ (but not $X$), and thus the spherical symmetry of standard Gaussian vectors ensures that $Y$ is still distributed as $\nu_{d-n-1}^d.$
\begin{proof}
For fixed $k$ and $d$, we begin by considering, for any two linear $k$-subspaces $A$ and $B$ of $\R^d$ that are orthogonal to $\mathbf{1}$, the collection $\Phi_{k,d}(A,B)$ of orthogonal $d\times d$ matrices which send $A$ to $B$ and fix $\mathbf{1}$.  As a subgroup of the orthogonal group, $\Phi_{k,d}(A,B)$ admits a Haar measure, via which we can define $\phi_{A,B}$ to be a \emph{random} orthogonal matrix mapping $A$ to $B$ and fixing $\mathbf{1}$.

    Now we give an algorithm to generate the coupled variables $(X,Y)$.  We begin by generating $X$ as 
    \[
    X=(I_d-\frac 1 d J_d)\Z_{d\times n}.
    \]For an independent copy $\Z'_{d\times n}$ of $\Z_d$, we generate a random matrix $W$ as
    \[
    W=(I_d-\frac 1 d J_d)\Z_{d\times n}.
    \]
    Note that with probability 1, all the columns of $X$ and $W$ are linearly independent.

    We append an all-1's column to $X$ to make a matrix $X'$ (by construction of $X$, this new column is orthogonal to the columns of $X$ and of $W$).  Now let $B$ denote the subspace spanned by the columns of $X'$, and $A$ denote the subspace orthogonal to the columns of $W$.  Note that with probability 1, $\dim(A)=\dim(B)=d+1$.  We choose the random $\phi_{A,B}$ and set $Y=\phi_{A,B} W$.  Observe that the columns of $Y$ are orthogonal to the columns of $X'$.
    
    We claim that $Y$ is also identical in distribution to $W$, so that $(X,Y)$ is indeed a coupling of $\nu_n^d$ and $\nu_{d-n-1}^d$.  Indeed, by the spherical symmetry of Gaussians, the subspace $B$ is uniformly random.  It is also independent of $W$.  Thus with $\phi_{A,B}$ chosen from the Haar measure on orthogonal transformations sending $A$ to $B$, $\phi_{A,B}$ is a uniformly random orthogonal transformation fixing $\mathbf{1}$, even after conditioning on any general position choice for $W$ (and thus $A)$.  Thus $Y$ has the same distribution as a uniformly random rotation of $W$ which fixes $\mathbf{1}$, and thus by Observation \ref{symmorth1}, has the same distribution of $W$, as claimed.
    \end{proof}

We finally prove the main theorem, Theorem \ref{t.equiv}.
\begin{proof}[Proof of Theorem \ref{t.equiv}]
    Fix $d$ and $k \leq \lfloor \frac{d + 1}{2} \rfloor$. For $d + 2$ points in $\R^d$, $x_1, \dots, x_{d + 2}$ in general position, the Gale dual is $d + 2$ points $y_1, \dots, y_{d + 2}$ on the real line with $y_1 + \cdots + y_{d + 2} = 0$. The Radon partition is $A \sqcup B$ so that $A = \{x_i \mid y_i < 0\}$ and $B = \{x_j \mid y_j > 0\}$. For Gaussian random points the event that $X_d = k$ is the event $\{|A| = k\} \cup \{|B| = k\}$. The probability that $|A| = k$ for $x_1,\dots, x_{d + 2}$ iid standard Gaussians in $\R^d$ is the same as for $(x_1,\dots, x_{d + 2}) \sim \rho_{d + 2}^d$ as $\rho_{d + 2}^d$ can be sampled as a translation by the barycenter from the standard Gaussian. By Lemma \ref{l.galegauss}, for $(x_1, \dots, x_{d + 2}) \sim \rho_{d + 2}^d$, $(y_1, \dots, y_{d+2}) \sim \rho_{d + 2}^1$. This means that $(y_1,\dots, y_{d+2})$ is sampled by taking $d + 2$ independent standard Gaussians on the real line and shifting by the barycenter. Thus $\Pr(|A| = k) = P(d + 2, k)$. So if $k \neq \frac{d + 1}{2}$, then $\Pr(X_d = k) = \Pr(|A| = k) + \Pr(|B| = k) = \Pr(|A| = k) + \Pr(|A| = d + 2 - k) = P(d + 2, k) + P(d + 2, d + 2 - k) = 2P(d + 2, k)$, and if $k = \frac{d + 1}{2}$ then $\Pr(X_d = k) = \Pr(|A| = k) = P(d + 2, k)$.
\end{proof}

\begingroup
\bibliography{References}
\bibliographystyle{amsplain}
\endgroup

\end{document}